\documentclass[11pt]{amsart}

\usepackage{fullpage,graphicx,amsfonts,amssymb,amsmath,amsthm}
\usepackage[all]{xy}
\usepackage[left=.8in,top=.8in,bottom=.8in,right=.8in,letterpaper]{geometry} 
\usepackage{mathtools}
\usepackage{graphicx}
\graphicspath{ {images/} }
\usepackage{enumerate}
\usepackage{setspace}
\usepackage{amssymb}

\allowdisplaybreaks
\theoremstyle{plain} %--default
\newtheorem{theorem}    {Theorem}

\newtheorem{lemma}      [theorem]{Lemma}

\theoremstyle{definition}
\newtheorem{definition} [theorem]{Definition}

\theoremstyle{remark}
\newtheorem{remark}              {Remark}

\numberwithin{equation}{section}
\numberwithin{theorem}{section}

\usepackage{url}

\begin{document}

\title{On Extensions of the Loomis-Whitney Inequality and Ball's Inequality for Concave, Homogeneous Measures}

\author{Johannes Hosle}
\address{Department of Mathematics, University of California, Los Angeles, CA 90095}
\email{jhosle@ucla.edu}

\begin{abstract}
The Loomis-Whitney inequality states that the volume of a convex body is bounded by the product of volumes of its projections onto orthogonal hyperplanes. We provide an extension of both this fact and a generalization of this fact due to Ball to the context of $q-$concave, $\frac{1}{q}-$homogeneous measures.
\end{abstract}
\maketitle

\section{Introduction}

The Loomis-Whitney inequality \cite{loomiswhitney} is a well-known geometric inequality concerning convex bodies, compact and convex sets with nonempty interior. Explicitly, the inequality states that if $u_1,...,u_n$ form an orthonormal basis of $\mathbb{R}^n$ and $K$ is a convex body in $\mathbb{R}^n$, then \begin{align*}
    |K|^{n-1} &\le \prod_{i=1}^{n}|K|u_i^{\perp}|,
\end{align*} where $K|u_i^{\perp}$ denotes the projection of $K$ onto $u_i^{\perp}$, the hyperplane orthogonal to $u_i$. Equality occurs if and only if $K$ is a box with faces parallel to the hyperplanes $u_i^{\perp}$. This was generalized by Ball \cite{ball}, who showed that if $u_1,...,u_m$ are vectors in $\mathbb{R}^n$ and $c_1,...,c_m$ positive constants such that \begin{align}
    \sum_{i=1}^{m} c_i u_i \otimes u_i = I_n,
\end{align} then \begin{align*}
    |K|^{n-1} &\le \prod_{i=1}^{m} |K|u_i|^{c_i}.
\end{align*} Here $u_i \otimes u_i$ denotes the rank $1$ projection onto the span of $u_i$, so $(u_i \otimes u_i)(x) = \langle x, u_i \rangle u_i$ with $\langle \cdot, \cdot \rangle$ representing the standard Euclidean inner product, and $I_n$ is the identity on $\mathbb{R}^n$. What will be useful later is the fact that \begin{align}\sum_{i=1}^{m} c_i = n,\end{align} which follows by comparing traces in (1.1).

The Loomis-Whitney inequality and Ball's inequality have been the subject of various generalizations. For instance, Huang and Li \cite{huangli} provided an extension of Ball's inequality with intrinsic volumes replacing volumes and an arbitrary even isotropic measure replacing the discrete measure $\sum_{i=1}^{m}c_i \delta_{u_i}$ in the condition $\int_{S^{n-1}} u \otimes u \ d\left(\sum_{i=1}^{m}c_i \delta_{u_i}\right)(u) = I_n$ of (1.1). They \cite{LpLM} also demonstrated the $L_p$ Loomis-Whitney inequality for even isotropic measures, while Lv \cite{Lv} very recently demonstrated the $L_{\infty}$ Loomis-Whitney inequality.

In this paper, we will first give a generalization of the original Loomis-Whitney inequality to the context of $q-$concave, $\frac{1}{q}-$homogeneous measures. Using a different argument, we shall then prove a generalization of Ball's inequality. Our two theorems are independent in the sense that the first is not recovered when specializing the second to the case of $u_1, ... u_n$ being an orthonormal basis and $c_1 = ... = c_n = 1$. Therefore, in fact, two different extensions of the Loomis-Whitney inequality are given. 

Let us recall the necessary definitions.

\begin{definition}
A function $f: \mathbb{R}^n \to [0,\infty]$ is $p-$concave for some $p \in \mathbb{R}\setminus \{0\}$ if for all $\lambda \in [0,1]$ and $x, y \in \text{supp}(f)$ we have \begin{align*}
    f(\lambda x + (1-\lambda) y) &\ge \left(\lambda f^p(x) + (1-\lambda)f^p(y) \right)^{\frac{1}{p}}.
\end{align*}
\end{definition}

\begin{definition}
A function $f: \mathbb{R}^n \to [0,\infty]$ is $r-$homogeneous if for all $a>0, x\in \mathbb{R}^n$ we have $f(ax) = a^r f(x)$.
\end{definition}

We will interested in the functions $g$ that are both $s-$concave for some $s>0$ and $\frac{1}{p}-$homogeneous for some $p>0$. In this case, we get that in fact $g$ is $p$-concave (see e.g. Livshyts \cite{livshyts}). Continuity will be assumed throughout. An example of a $p-$concave, $\frac{1}{p}-$homogeneous function is $g(x) = 1_{\langle x, \theta \rangle > 0} \langle x, \theta \rangle^{\frac{1}{p}}$, where $\theta$ is a vector. All such functions $g$, with the exception of constant functions, will be supported on convex cones. To see this, observe that concavity implies that the support is convex and homogeneity implies that if $x \in \text{supp}(g)$ then $tx \in \text{supp}(g)$ for all $t>0$. Moreover, we cannot have both $x, -x \in \text{supp}(g)$, for then concavity will give $g(0) = g\left(\frac{1}{2}x + \frac{1}{2}(-x)\right) > 0$, but $g(0) = 0$ by homogeneity. 

A notation we will use is $\tilde{g}(x) = g(x) + g(-x).$

If $\mu$ is a measure with a $p-$concave, $\frac{1}{p}-$homogeneous density, then a change of variables will show that $\mu$ is $n+\frac{1}{p}$ homogeneous, that is $\mu(tK) = t^{n+\frac{1}{p}} \mu(K)$. From a result of Borell \cite{borell}, we also have concavity: 

\begin{lemma}[Borell]
Let $p \in \left( -\frac{1}{n}, \infty\right]$ and let $\mu$ be a measure on $\mathbb{R}^n$ with $p-$concave density $g$. For $q = \frac{1}{n+\frac{1}{p}}$, $\mu$ is a $q-$concave measure, that is for measurable sets $E, F$ and $\lambda \in [0,1]$ we have \begin{align*}
    \mu(\lambda E + (1-\lambda)F) &\ge (\lambda \mu(E)^{q}+(1-\lambda)\mu(F)^q)^{\frac{1}{q}}.
\end{align*}
\end{lemma}

To now define the generalized notion of projection for measures, one requires the definition of mixed measures (see e.g. Livshyts \cite{livshyts}). 
\begin{definition}
Let $A, B$ be measurable sets in $\mathbb{R}^n$. We define \begin{align*}
    \mu_1(A, B) = \liminf_{\varepsilon \to 0} \frac{\mu(A+\varepsilon B) - \mu(A)}{\varepsilon}
\end{align*} to be the mixed $\mu-$measure of $A$ and $B$.
\end{definition}

An important simple fact, which follows from Lemma 3.3 in Livshyts \cite{livshyts}, is that mixed measure is linear in the second variable, so \begin{align}
    \mu_1(K, E + tF) = \mu_1(K, E) + t \mu_1(K, F)
\end{align} for $t\ge 0$.

For $q-$concave measures, we have the following generalization of Minkowski's first inequality (see e.g. Milman and Rotem \cite{milmanrotem}):

\begin{lemma}
Let $\mu$ be a $q-$concave measure and $A, B$ be measurable sets in $\mathbb{R}^n$. Then, \begin{align*}
    \mu(A)^{1-q} \mu(B)^q &\le q \mu_1(A, B).
\end{align*}
\end{lemma}

We now turn to discussing the generalized notion of projection. This notion, defined by Livshyts \cite{livshyts}, is \begin{align}
    P_{\mu, K}(\theta) = \frac{n}{2} \int_{0}^{1} \mu_1(tK, [-\theta, \theta]) dt
\end{align}  for $\theta \in S^{n-1}$, where $K$ is a convex body, $\mu$ is an absolutely continuous measure, and $[-\theta, \theta] = \{t\theta: t\in [-1,1]\}$. This is a natural extension of the identity $|K|\theta^{\perp}| = \frac{1}{2}\lambda_1(K, [-\theta, \theta])$, with $\lambda$ denoting Lebesgue measure, which can be readily seen for polytopes and follows in the general case by approximation. 

In \cite{livshyts}, a version of the Shephard problem for $q-$concave, $\frac{1}{q}-$homogeneous measures was proven with this notion of measure. The author in \cite{hosle} studied the related section and projection comparison problems, including for this same class of $q-$concave, $\frac{1}{q}-$homogeneous measures. 

With (1.4), we can now state our first theorem:

\begin{theorem}
Let $\mu$ be a measure with $p-$concave, $\frac{1}{p}-$homogeneous density $g$ for some $p>0$. Then, for any convex body $K$ and an orthonormal basis $(u_i)_{i=1}^n$ with $[-u_i, u_i] \cap \text{supp}(g) \neq \varnothing$ for each $1\le i\le n$, \begin{align*}
    \mu(K)^{n+\frac{1}{p}-1} \le 2^{n+\frac{1}{p}} \left(1+\frac{1}{pn}\right)^{n}\left( \sum_{k=1}^{n}\tilde{g}^p(u_k) \right)^{-\frac{1}{p}} \prod_{i=1}^{n} P_{\mu, K}(u_i)^{1+ \frac{\tilde{g}^p(u_i)}{p \sum_{k=1}^{n}\tilde{g}^p(u_k) }}.
\end{align*}
\end{theorem}

Before we state our generalization of Ball's inequality, we introduce another definition. Let $\mathcal{S} = \{(u_i)_{i=1}^{m}\}$ be a set of unit vectors in $\mathbb{R}^n$. Then we define $\mathcal{S}^{(1)}$ to be the set of $u_{ij} = \frac{u_i - \langle u_i, u_j \rangle u_j}{|u_i - \langle u_i, u_j \rangle u_j|},$ the normalized projection of $u_i$ onto the hyperplane $u_j^{\perp}$, for $1\le i, j\le m.$ Recursively defining $S^{(k)} = (S^{(k-1)})^{(1)}$, we set
\begin{align}
    \mathcal{P} = \mathcal{P}((u_i)_{i=1}^{m}) := \mathcal{S} \cup \mathcal{S}^{(1)} \cup ... \cup \mathcal{S}^{(n-1)},
\end{align} some finite sets depending on our initial choice of $\{(u_i)_{i=1}^{m}\}$. Our generalization of Ball's inequality is the following:

\begin{theorem}
Let $\mu$ be a measure with $p-$concave, $\frac{1}{p}-$homogeneous density $g$ for some $p>0$. If $(u_i)_{i=1}^{m}$ are unit vectors in $\mathbb{R}^n$ and $(c_i)_{i=1}^{m}$ are positive constant such that \begin{align*}
    \sum_{i=1}^{m} c_i u_i \otimes u_i = I_n
\end{align*} and moreover $[-u, u] \cap \text{supp}(g) \neq \varnothing$ for each $u \in \mathcal{P}((u_i)_{i=1}^{m})$, then \begin{align*}
    \mu(K)^{n+\frac{1}{p}-1} &\le 2^{n+\frac{1}{p}} \left(\inf_{u \in \mathcal{P}}\tilde{g}(u)\right)^{-1} \prod_{k=1}^{n}\left(1+\frac{1}{kp}\right) \prod_{i=1}^{m} P_{\mu, K}(u_i)^{c_i\left(1 + \frac{1}{pn}\right)}
\end{align*} for any convex body $K$.
\end{theorem}

Observe that the condition $[-u, u] \cap \text{supp}(g) \neq \varnothing$ is not particularly restrictive. For instance, if we consider $g$ whose support is a half space with boundary a half plane $P$, then the condition simply reduces to the fact that some finite number of points do not lie on $P$. 

\begin{remark}
Consider $g(x) = 1_{\langle x, \theta\rangle>0} \langle x, \theta\rangle^{\frac{1}{p}}$ where either $u_i \not\in \theta^{\perp}$ for $1\le i\le n$ with the assumptions of Theorem 1.6 or $u \not\in \theta^{\perp}$ for each $u \in \mathcal{P}((u_i)_{i=1}^{m})$ in the assumptions of Theorem 1.7. Then, taking $p\to\infty$, Theorem 1.6 and Theorem 1.7 recover the results for Lebesgue measure up to a dimensional constant of $2^n$. The reason for this extra factor of $2^n$ comes from the fact that nonconstant $p-$concave, $\frac{1}{p}-$homogeneous densities are supported on at most a half-space, which therefore restricts us to only being able to get inequalities on 'half' of our domain. 
\end{remark}

\textbf{Acknowledgements. }I am very grateful to Galyna Livshyts and Kateryna Tatarko for helpful discussions on this topic and comments on this manuscript. I would also like to thank the anonymous referee for comments that improved the exposition of this paper.

\begin{section}{Extension of the Loomis-Whitney Inequality}

We begin with a lemma providing us with a lower bound for the measure of a face of a parallelapiped. With homogeneity, this will give us a lower bound for the measure of a parallelapiped, which will be a key ingredient in the proof of Theorem 1.6.
%In this method, we shall estimate the measures of the faces directly, without first projecting them as was done in the previous section.

\begin{lemma}
Let $g, \mu, (u_i)_{i=1}^{n}$ be as in the statement of Theorem 1.6, let $$F_i = \{u = \alpha_i u_i + \sum_{j\neq i}\beta_j u_j: |\beta_j| \le \alpha_j \},$$ where $\alpha_1,..,\alpha_n$ are positive constants, and suppose that $u_i \in \text{supp}(g)$. Then, \begin{align*}
    \mu_{n-1}(F_i) \ge \left(\frac{pn}{pn+1}\right)^n\left(1+\frac{\tilde{g}^p(u_i)}{p \sum_{k=1}^{n}\tilde{g}^p(u_k)}\right)\left(\sum_{i=1}^{n}\tilde{g}^p(u_i)\right)^{\frac{1}{p}} \alpha_i^{-1} \prod_{j=1}^{n} \alpha_j^{1+\frac{\tilde{g}^p(u_j)}{p\sum_{i=1}^{n}\tilde{g}^p(u_i)}},
\end{align*} where $\mu_{n-1}(F_i)$ denotes the integral of $g$ over the $(n-1)-$dimensional set $F_i$.
\end{lemma}
\begin{proof}
For simplicity of notations, we deal with the case $i=1$. We begin by writing $\mu_{n-1}(F_1)$ as an integral of $g$ over $F_1$, subdividing the domain of integration, and using homogeneity: \begin{align*}
    \mu_{n-1}(F_1) &:= \int_{\substack{\text{$v = \alpha_1 u_1 + \sum_{j=2}^{n}\beta_j u_j$ } \\ \text{$|\beta_j| \le \alpha_j$}}} g(v) dv \\ &= \sum_{\sigma = (\pm 1, ..., \pm 1)} \int_{0}^{\alpha_n}...\int_{0}^{\alpha_2} g\left(\alpha_1 u_1 + \sum_{j=2}^{n} \beta_j \sigma(j) u_j \right) d\beta_2 ... d\beta_n \\ &= \sum_{\sigma = (\pm 1,...,\pm 1)} \int_{0}^{\alpha_n}...\int_{0}^{\alpha_2} \left( \alpha_1 + \sum_{j=2}^{n}\beta_j\right)^{\frac{1}{p}} g\left(\frac{\alpha_1}{\alpha_1 + \sum_{j=2}^{n}\beta_j} u_1 + \sum_{j=2}^{n} \frac{\beta_j}{\alpha_1 + \sum_{j=2}^{n}\beta_j} \sigma(j) u_j \right) d\beta_2 ... d\beta_n \\ &= \sum_{\sigma = (\pm 1,...,\pm 1)} I_{\sigma}.
\end{align*} If we take $\sigma'$ such that ${\sigma}'(j)u_j \in \text{supp}(g)$ for each $j$ (which can be done by the hypothesis of Theorem 1.6), then \begin{align}
    \mu_{n-1}(F_1) &\ge I_{{\sigma}'}.
\end{align}

By $p-$concavity and the fact that $g(\sigma'(j)u_j) = \tilde{g}(u_j)$,
\begin{align*}
    I_{\sigma'} &\ge \int_{0}^{\alpha_n}...\int_{0}^{\alpha_2} \left(\alpha_1+\sum_{j=2}^{n} \beta_j\right)^{\frac{1}{p}} \left(\frac{\alpha_1}{\alpha_1+\sum_{j=2}^{n}\beta_j}\tilde{g}^p(u_1) + \sum_{j=2}^{n} \frac{\beta_j}{\alpha_1+\sum_{j=2}^{n}\beta_j} \tilde{g}^p(u_j) \right)^{\frac{1}{p}} d\beta_2 ... d\beta_n \\ &= \int_{0}^{\alpha_n} ... \int_{0}^{\alpha_2} \left(\alpha_1 \tilde{g}^p(u_1) + \sum_{j=2}^{n}\beta_j \tilde{g}^p(u_j) \right)^{\frac{1}{p}} d\beta_2 ... d\beta_n \\ &= \left(\sum_{i=1}^{n} \tilde{g}^p(u_i)\right)^{\frac{1}{p}} \int_{0}^{\alpha_n} ... \int_{0}^{\alpha_2} \left(\alpha_1 \frac{\tilde{g}^p(u_1)}{\sum_{i=1}^{n} \tilde{g}^p(u_i)} + \sum_{j=2}^{n} \beta_j \frac{\tilde{g}^p(u_j)}{\sum_{i=1}^{n} \tilde{g}^p(u_i)} \right)^{\frac{1}{p}} d\beta_2 ... d\beta_n.
\end{align*} Inserting the bound \begin{align*}
\alpha_1 \frac{\tilde{g}^p(u_1)}{\sum_{i=1}^{n} \tilde{g}^p(u_i)} + \sum_{j=2}^{n} \beta_j \frac{\tilde{g}^p(u_j)}{\sum_{i=1}^{n} \tilde{g}^p(u_i)} &\ge \alpha_1^{\frac{\tilde{g}^p(u_1)}{\sum_{i=1}^{n} \tilde{g}^p(u_i)}} \prod_{j=2}^{n}\beta_j^{\frac{\tilde{g}^p(u_j)}{\sum_{i=1}^{n} \tilde{g}^p(u_i)}}
\end{align*} from the arithmetic mean-geometric mean inequality under the integral gives \begin{align*}
    I_{\sigma'} &\ge \left(\sum_{i=1}^{n}\tilde{g}^p(u_i)\right)^{\frac{1}{p}}\alpha_1^{\frac{\tilde{g}^p(u_1)}{p\sum_{i=1}^{n} \tilde{g}^p(u_i)}} \prod_{j=2}^{n} \frac{1}{1+\frac{\tilde{g}^p(u_j)}{p\sum_{i=1}^{n}\tilde{g}^p(u_i)}} \alpha_j^{1+\frac{\tilde{g}^p(u_j)}{p\sum_{i=1}^{n}\tilde{g}^p(u_i)}} \\ &= \left(1+\frac{\tilde{g}^p(u_1)}{p \sum_{i=1}^{n}\tilde{g}^p(u_i)}\right)\left(\sum_{i=1}^{n}\tilde{g}^p(u_i)\right)^{\frac{1}{p}} \alpha_1^{-1} \prod_{j=1}^{n}\frac{1}{1+\frac{\tilde{g}^p(u_j)}{p\sum_{i=1}^{n}\tilde{g}^p(u_i)}} \alpha_j^{1+\frac{\tilde{g}^p(u_j)}{p\sum_{i=1}^{n}\tilde{g}^p(u_i)}}.
\end{align*} Again by the arithmetic mean-geometric mean inequality, \begin{align*}
    \prod_{j=1}^{n} \left(1+\frac{\tilde{g}^p(u_j)}{p\sum_{i=1}^{n}\tilde{g}^p(u_i)}\right) &\le \left(1+\frac{1}{pn}\right)^n,
\end{align*} and thus \begin{align*}
    I_{\sigma'} &\ge \left(\frac{pn}{pn+1}\right)^n\left(1+\frac{\tilde{g}^p(u_1)}{p \sum_{i=1}^{n}\tilde{g}^p(u_i)}\right)\left(\sum_{i=1}^{n}\tilde{g}^p(u_i)\right)^{\frac{1}{p}} \alpha_1^{-1} \prod_{j=1}^{n} \alpha_j^{1+\frac{\tilde{g}^p(u_j)}{p\sum_{i=1}^{n}\tilde{g}^p(u_i)}}.
\end{align*} By (2.1), our proof is complete.

\end{proof}

For the proof of our theorem, we will recall the definition of a zonotope. A zonotope is simply a Minkowski sum of line segments \begin{align*}
    Z= \sum_{i=1}^{m}[-x_i, x_i].
\end{align*} By linearity (1.3), if $Z = \sum_{i=1}^{m} \alpha_i [-u_i, u_i]$ for unit vectors $u_i$ and $\alpha_i$ positive constants, then \begin{align*}
    \mu_1(K,Z) = \sum_{i=1}^{m} \alpha_i \mu_1(K,[-u_i, u_i])
\end{align*} for a convex body $K$. Since our measure $\mu$ is homogeneous, \begin{align*}
P_{\mu,K}(u_i) &= \frac{n}{2} \int_{0}^{1} \mu_1(tK, [-u_i, u_i]) dt \\ &= \frac{n}{2} \int_{0}^{1} t^{\frac{1}{q}-1} dt \mu_1(K, [-u_i, u_i]) \\ &= \frac{qn}{2} \mu_1(K, [-u_i, u_i])
\end{align*} by (1.4). Therefore, \begin{align}
    \mu_1(K, Z) &= \frac{2}{nq}\sum_{i=1}^{m} \alpha_i P_{\mu, K}(u_i).
\end{align}

We now prove our theorem:
\begin{proof}[Proof of Theorem 1.6. ]
Let $Z$ be the zonotope $\sum_{i=1}^{n} \alpha_i [-u_i, u_i]$ with $\alpha_i = \frac{1}{P_{\mu, K}(u_i)}$ for $1\le i \le n$. By Lemma 1.5, (2.2), and our choice of $\alpha_i$, \begin{align*}
    \mu(K)^{1-q} &\le q\mu(Z)^{-q} \mu_1(K, Z) \\ &= 2\mu(Z)^{-q},
\end{align*} 
and so \begin{align}
    \mu(K)^{\frac{1}{q}-1} &\le 2^{\frac{1}{q}} \mu(Z)^{-1}.
\end{align}

Without loss of generality, we assume that $u_i \in \text{supp}(g)$ and $g(-u_i) = 0$ for each $i$. Let $F_i$ denote the face of $Z$ orthogonal to and touching $\alpha_i u_i$, and subdivide $Z$ into pyramids with bases of $F_i$, apex at the origin, and height of $\alpha_i$. By homogeneity, \begin{align*}
    \mu(Z) &= \sum_{i=1}^{n} \int_{0}^{\alpha_i} \mu_{n-1}\left(\frac{t}{\alpha_i}F_i\right) dt \\ &= \sum_{i=1}^{n} \left(\int_{0}^{\alpha_i} t^{\frac{1}{q}-1} dt\right) \alpha_i^{1-\frac{1}{q}} \mu_{n-1}(F_i) \\ &= q \sum_{i=1}^{n} \alpha_i \mu_{n-1}(F_i).
\end{align*} Applying Lemma 2.1, we have \begin{align*}
    \mu(Z) &\ge \frac{1}{n+\frac{1}{p}} \left(\frac{pn}{pn+1}\right)^n\left(\sum_{i=1}^{n}\tilde{g}^p(u_i)\right)^{\frac{1}{p}}\left( \prod_{j=1}^{n} \alpha_j^{1+\frac{\tilde{g}^p(u_j)}{p\sum_{i=1}^{n}\tilde{g}^p(u_i)}}\right) \sum_{i=1}^{n}\left( 1+\frac{\tilde{g}^p(u_i)}{p \sum_{k=1}^{n}\tilde{g}^p(u_k)} \right) \\ &= \left(\frac{pn}{pn+1}\right)^n\left(\sum_{i=1}^{n}\tilde{g}^p(u_i)\right)^{\frac{1}{p}} \prod_{j=1}^{n} \alpha_j^{1+\frac{\tilde{g}^p(u_j)}{p\sum_{i=1}^{n}\tilde{g}^p(u_i)}}.
\end{align*} Combining this bound with (2.3) and recalling that $\alpha_i = \frac{1}{P_{\mu, K}(u_i)}$, our desired inequality is proven.

\end{proof}
\end{section}

\begin{section}{Extension of Ball's Inequality}

As in the previous section, we will require an estimate from below for the measure of a zonotope. However, mimicking the approach of Ball \cite{ball}, rather than estimating the measures of the faces directly, we shall first project them.  A main difference from Ball's proof stems from the lack of translation invariance of our measure, but we will circumvent this obstacle by an appropriate inequality (3.2) coming from concavity.

\begin{lemma}
Let $g, \mu, (u_i)_{i=1}^{m}, (c_i)_{i=1}^{m}$ be as in the statement of Theorem 1.7. Let $Z = \sum_{i=1}^{m}\alpha_i [-u_i, u_i]$ be a zonotope. Then \begin{align*}
    \mu(Z) &\ge \left(\inf_{u \in \mathcal{P}} \tilde{g}(u)\right) \left(\prod_{k=1}^{n} \frac{k}{k+\frac{1}{p}}\right) \prod_{i=1}^{m} \left(\frac{\alpha_i}{c_i} \right)^{c_i\left(1+\frac{1}{pn}\right)}.
\end{align*}
\end{lemma}
\begin{proof}
Following Ball \cite{ball}, we induct on the dimension $n$. First consider the case $n=1$. We can then assume $u_1 = ... = u_m$ and without loss of generality $g(u_1) = \tilde{g}(u_1) > 0$ and $g(-u_1) = 0$. Then \begin{align*}
    \mu(Z) &= \mu\left(\left(\sum_{i=1}^{m}\alpha_i\right)[-u_1, u_1] \right) \\ &= \int_{0}^{\sum_{i=1}^{m}\alpha_i} g(t u_1) dt \\ &= \left(\int_{0}^{\sum_{i=1}^{m}\alpha_i} t^{\frac{1}{p}} dt\right)g(u_1) \\ &= \frac{1}{1+\frac{1}{p}} \left(\sum_{i=1}^{m}\alpha_i \right)^{1+\frac{1}{p}} g(u_1).
\end{align*} Since $n=1$, (1.2) implies $\sum_{i=1}^{m} c_i = 1$, and therefore by the arithmetic mean-geometric mean inequality \begin{align*}
    \sum_{i=1}^{m} \alpha_i &= \sum_{i=1}^{m} c_i \frac{\alpha_i}{c_i} \ge \prod_{i=1}^{m} \left(\frac{\alpha_i}{c_i}\right)^{c_i}.
\end{align*} This concludes the proof for $n=1$. 

Let us assume we now have our result for dimension $n-1$, and consider the case of dimension $n$. Firstly, observe that homogeneity implies \begin{align*}
\mu_1(Z,Z) &= \liminf_{\varepsilon \to 0} \frac{\mu(Z + \varepsilon Z) - \mu(Z)}{\varepsilon} \\ &= \liminf_{\varepsilon \to 0} \mu(Z) \frac{(1+\varepsilon)^{\frac{1}{q}} - 1}{\varepsilon} \\ & = \frac{1}{q} \mu(Z).
\end{align*} 
Therefore, \begin{align*}
    \mu(Z) &= q\mu_1(Z, Z) \\ &= q \sum_{i=1}^{m} \alpha_i \mu_1(Z, [-u_i, u_i]) \\ &= q n \sum_{i=1}^{m} \frac{c_i}{n} \frac{\alpha_i}{c_i} \mu_1(Z, [-u_i, u_i]).
\end{align*} Since $\sum_{i=1}^{m}\frac{c_i}{n} = 1$, we use the arithmetic mean-geometric mean inequality once again to get \begin{align}
    \mu(Z) &\ge qn\prod_{i=1}^{m}\left(\frac{\alpha_i}{c_i}\mu_1(Z, [-u_i, u_i])\right)^{\frac{c_i}{n}}.
\end{align}

Let $P_i Z$ denote the projection of $Z$ onto the hyperplane $u_i^{\perp}$. We wish to show \begin{align}
    \mu_1(Z, [-u_i, u_i]) &\ge \mu_{n-1}(P_i Z),
\end{align} where $\mu_{n-1}$ denotes integration of the density $g$ over the $(n-1)-$dimensional set $P_i Z$. This will compensate for the lack of translation invariance of our measure.

By assumption, one of $u_i$ and $-u_i$ lies in $\text{supp}(g)$. Without loss of generality, $u_i \in \text{supp}(g)$. For $w \in \mathbb{R}^n$ and $t>0$, concavity and homogeneity give us \begin{align*}
    g(w+tu_i) &\ge \left(g^p(w) + t g^p(u_i)\right)^{\frac{1}{p}} \ge g(w).
\end{align*} To be precise, concavity gives this to us when $w \in \text{supp}(g)$, but when $w \not\in \text{supp}(g)$ this is trivial. This inequality is equivalent to the statement that \begin{align}
g(w+t_1 u_i) &\ge g(w + t_2 u_i)    
\end{align} for any $w\in \mathbb{R}^n$ and $t_1 \ge t_2$.

For each $w \in P_i Z$, let $t(w)\ge 0$ be taken so that $w + t(w) u_i \in \partial Z$. We now write \begin{align*}
    \mu_1(Z, [-u_i, u_i]) &= \liminf_{\varepsilon\to 0} \frac{\mu(Z + \varepsilon [-u_i, u_i]) - \mu(Z)}{\varepsilon} \\ &= \liminf_{\varepsilon \to 0} \frac{\mu((Z+\varepsilon[-u_i,u_i])\setminus Z)}{\varepsilon} \\ &\ge \liminf_{\varepsilon\to 0}\frac{\mu((Z+\varepsilon[0,u_i])\setminus Z)}{\varepsilon} \\ &= \liminf_{\varepsilon \to 0} \frac{1}{\varepsilon} \int_{P_i Z} \int_{t(h)}^{t(h)+\varepsilon} g(h + s u_i) ds dh,
\end{align*} where our integral of the density is taken over the region $(Z+[0,u_i])\setminus Z$. By (3.3) and continuity,  
\begin{align*} \liminf_{\varepsilon \to 0} \frac{1}{\varepsilon} \int_{P_i Z} \int_{t(h)}^{t(h)+\varepsilon} g(h + s u_i) ds dh &\ge \liminf_{\varepsilon \to 0} \frac{1}{\varepsilon} \int_{P_i Z} \int_{0}^{\varepsilon} g(h+s u_i) ds dh \\ &= \mu_{n-1}(P_i Z).
\end{align*}This proves (3.2).

Denoting the projection of $u_j$ onto $u_i^{\perp}$ by $P_i(u_j)$, we have that $P_i Z$ is the zonotope \begin{align*}
    P_i Z &= \sum_{j=1}^{m} \alpha_j [-P_{i}(u_j), P_{i}(u_j)] \\ &= \sum_{i=1}^{m} \alpha_i \gamma_{ji} [-u_{ji}, u_{ji}],
\end{align*} where $\gamma_{ji} = |u_j - \langle u_i, u_j\rangle u_i|.$ A simple computation shows $\gamma_{ji}^2 = 1 - \langle u_i, u_j \rangle^2$. 

We also have \begin{align*}
    P_{i} &= \sum_{j=1}^{m} c_j P_{i}u_j \otimes P_{i} u_j \\ &= \sum_{j=1}^{m} \gamma_{ji}^2 c_j u_{ji} \otimes u_{ji},
\end{align*} and this is the identity operator on $u_i^{\perp}$. By (3.1), (3.2), and our inductive hypothesis, 
\begin{align*}
    \mu(Z) &\ge \frac{n}{n+\frac{1}{p}} \prod_{i=1}^{m} \left(\frac{\alpha_i}{c_i} \mu_{n-1}(P_{i}Z) \right)^{\frac{c_i}{n}}\\ &\ge \prod_{k=1}^{n} \frac{k}{k+\frac{1}{p}} \prod_{i=1}^{m} \left(\frac{\alpha_i}{c_i} \left(\inf_{u \in \mathcal{P}((u_{ji})_{j=1}^{m})}\tilde{g}(u)\right) \prod_{j=1}^{m}\left(\frac{\alpha_j \gamma_{ji}}{c_j \gamma_{ji}^2}\right)^{c_j \gamma_{ji}^2\left(1+\frac{1}{p(n-1)}\right)} \right)^{\frac{c_i}{n}} \\ &\ge \left(\inf_{u \in \mathcal{P}} \tilde{g}(u)\right) \left(\prod_{k=1}^{n}\frac{k}{k+\frac{1}{p}}\right) \prod_{i, j =1}^{m} \left( \left(\frac{\alpha_i}{c_i}\right)^{c_i} \left(\frac{\alpha_j}{c_j \gamma_{ji}}\right)^{c_ic_j \gamma_{ji}^2\left(1+\frac{1}{p(n-1)}\right)} \right)^{\frac{1}{n}}.
\end{align*} From the inequality $\frac{1}{\gamma_{ji}} \ge 1$ and the relation \begin{align*}
    \sum_{i=1}^{m} c_i \gamma_{ji}^2 &= \sum_{i=1}^{m}c_i(1- \langle u_i, u_j\rangle^2) = n-1,
\end{align*} an appropriate grouping of elements in our product completes the proof.
\end{proof}

As before, the proof of Theorem 1.7 now follows:
\begin{proof}[Proof of Theorem 1.7. ]
Let $Z$ be the zonotope $\sum_{i=1}^{m} \alpha_i [-u_i, u_i]$ where $\alpha_i = \frac{c_i}{P_{\mu, K}(u_i)}$ for $1\le i\le m$. By the same argument as in the proof of Theorem 1.6, where we must use (1.2), \begin{align*}
    \mu(K)^{\frac{1}{q}-1} &\le 2^{\frac{1}{q}} \mu(Z)^{-1}.
\end{align*} By Lemma 3.1, we reach \begin{align*}
    \mu(K)^{\frac{1}{q}-1} &\le 2^{\frac{1}{q}} \left(\inf_{u \in \mathcal{P}} \tilde{g}(u) \right)^{-1} \prod_{k=1}^{n} \left(1+\frac{1}{kp}\right) \prod_{i=1}^{m} P_{\mu, K}(u_i)^{c_i \left(1+\frac{1}{pn}\right)}
\end{align*} as desired.
\end{proof}

\end{section}

\bibliographystyle{alpha}

\end{document}